\documentclass{amsart} 
\usepackage{amssymb}
\usepackage{amsthm}
\usepackage{amsmath}
\usepackage[curve,matrix,arrow]{xy}
\theoremstyle{plain}\newtheorem{Theorem}{Theorem}[section]
\theoremstyle{plain}
\theoremstyle{plain}\newtheorem{Corollary}[Theorem]{Corollary}
\theoremstyle{plain}\newtheorem{Lemma}[Theorem]{Lemma}
\theoremstyle{plain}\newtheorem{Proposition}[Theorem]{Proposition}
\theoremstyle{definition}
\theoremstyle{definition}
\theoremstyle{definition}
\theoremstyle{definition}\newtheorem{Remark}[Theorem]{Remark}
\theoremstyle{definition}
\theoremstyle{plain}

    \def\OG{{\mathcal{O}G}}  \def\OGb{{\mathcal{O}Gb}}
    \def\OH{{\mathcal{O}H}}  
    \def\OP{{\mathcal{O}P}}
    \def\OQ{{\mathcal{O}Q}}
    \def\OR{{\mathcal{O}R}}
\def\CF{{\mathcal{F}}}

\def\CO{{\mathcal{O}}}

\def\Q{{\mathbb Q}}               
\def\Z{{\mathbb Z}}

\def\Aut{\mathrm{Aut}}                
\def\Br{\mathrm{Br}}             \def\ten{\otimes}

\def\IBr{\mathrm{IBr}}
             \def\tenB{\otimes_B}
\def\Ind{\mathrm{Ind}}

\def\Irr{\mathrm{Irr}}           

           \def\tenO{\otimes_{\mathcal{O}}}

           \def\tenOP{\otimes_{\mathcal{O}P}}

\def\Res{\mathrm{Res}}           \def\tenOR{\otimes_{\mathcal{O}R}}

\title{On the $p$-part of the conductor of a generalised character} 
\author{Markus Linckelmann} 
\date{23.3. 2026}

\address{Markus Linckelmann \\
School of Science \& Technology \\
Department of Mathematics \\
City St. George's, University of London \\
Northampton Square \\
London EC1V 0HB \\
United Kingdom}
\email{markus.linckelmann.1@city.ac.uk}

\subjclass[2010]{20C15, 20C20, 20J05}

\keywords{Finite group, block, isotypy, generalised character, conductor}

\begin{document}

\maketitle  

\begin{abstract}
We show that the $p$-part of the conductor of a generalised  character of a finite
group is equal to the conductor of its generalised decomposition numbers.
We use this to show that $p$-parts of conductors of irreducible characters 
are preserved under isotypies and perfect isometries that arise in the context of  
stable  equivalences of Morita type with endopermutation source. 
 We  apply this  to blocks with abelian defect and Frobenius inertial quotient. 
\end{abstract}

\section{Introduction} 

Let $p$ be a prime, $\CO$ a complete discrete valuation ring with residue
field $k$ of characteristic $p$ and field of fractions $K$ of characteristic zero.
For any non-negative integer $n$ we denote by $\zeta_n$ a primitive $n$-th
root of unity in an extension field of $K$.

\medskip
Let $G$ be a finite group. Assume that $K$ contains a primitive $|G|$-th root 
of unity. Denote by $\Irr(G)$ the set of $K$-valued irreducible characters of $G$
and by $\IBr(G)$ the set of $K$-valued irreducible Brauer characters. 
We denote by $G_p$ and $G_{p'}$ the sets of elements of $G$ of order
a power of $p$ and of order prime to $p$, respectively.
 Let $u\in $ $G_p$. By results of Brauer, for every generalised character
 $\chi\in$ $\Z\Irr(G)$ and any $\varphi\in$ $\IBr(G)$ there are unique cyclotomic integers
$d^u_{\chi,\varphi}$  such that for  every $s\in C_G(u)_{p'}$  we have 
$$\chi(us) = \sum_{\varphi\in\IBr(C_G(u))}\ d^u_{\chi,\varphi} \ \varphi(s).$$
The cyclotomic integers $d^u_{\chi,\varphi}$ are 
contained in $\Q(\eta)$ for some root of unity $\eta$ of $p$-power order
dividing the order of $u$. For an expository account of these facts see
e.g. \cite[\S 5.15]{LiBookI}. Following the notation and terminology in
\cite{NT21}, \cite{HungFry}, 
given a  set $S$ of cyclotomic integers, we denote by $c(S)$ the
smallest positive integer $n$ such that $S\subseteq$ $\Q(\zeta_n)$.
If $S$ consists of a single cyclotomic integer $\sigma$, then we write
$c(\sigma)$ instead of $c(\{\sigma\})$. 

\medskip
For $\chi\in$ $\Z\Irr(G)$, we set $c(\chi)=$ $c(\{\chi(g)\ |\ g\in G\})$; this
integer is called the {\em conductor of} $\chi$. 
We denote by $c(\chi)_p$ the $p$-part of $c(\chi)$; that is, $c(\chi)_p$ is
the highest power $p^a$ of $p$ that divides $c(\chi)$. The integer $a$ is called
{\em $p$-rationality level of $\chi$} in  \cite{HungFry}.  
The present paper has been motivated by the question, raised by
Hung Ngoc Nguyen, whether perfect isometries can be used 
for calculating conductors of characters. We show in this paper
that character bijections induced by isotypies preserve the $p$-parts of conductors.
This is based on the  following result which shows that $c(\chi)_p$ is determined 
by the non-ordinary generalised decomposition numbers of $\chi$. 

\begin{Theorem}\label{thm1}
Let $G$ be a finite group. Assume that $K$ contains a primitive $|G|$-th root 
of unity.  For any $\chi\in$ $\Z\Irr(G)$ we have 
$$c(\chi)_p = c\left(\{d^u_{\chi,\varphi}\ |\ u\in G_p\ \text{and}\ 
\varphi\in\IBr(C_G(u))\}\right).$$
\end{Theorem}

This will be proved in Section \ref{thm1-proof}. Since the numbers $d^u_{\chi,\varphi}$ in
Theorem \ref{thm1} involve only $p$-power roots of unity, it follows that
$c(\chi)_p$ is in fact the maximum of the numbers $c(d^u_{\chi,\varphi})$, with
$u\in$ $G_p$ and $\varphi\in$ $\IBr(C_G(u))$.

\begin{Corollary}\label{Cor05}
Let $G$ be a finite group. Assume that $K$ contains a primitive $|G|$-th root 
of unity.  Let $\chi\in$ $\Z\Irr(G)$. There is $u\in G_p$ and $\varphi\in$
$\IBr(C_G(u))$ such that $c(\chi)_p=c(d^u_{\chi,\varphi})$ and such that
$c(\chi)_p \geq c(d^{v}_{\chi, \psi})$ for all $v\in G_p$ and
all $\psi \in$ $\IBr(C_G(v))$
\end{Corollary}

As a consequence of results due to Puig,  the  generalised decomposition numbers 
of a block, and hence $p$-parts of conductors of characters, 
can be read off its source  algebras (see \cite[Remark 6.13.10]{LiBookII}).
Thus a source algebra isomorphism, or equivalently, a Morita equivalence given
by a bimodule with trivial source,   preserves the $p$-parts of conductors of
characters corresponding to each other through such a Morita equivalence.
This conclusion holds more generally for Morita equivalences with endopermutation
sources.

\begin{Corollary} \label{Cor1}
Let $B$, $B'$ be blocks of finite group algebras $\OG$, $\OG'$, respectively,
and let $M$ be a $B'$-$B$-bimodule inducing a Morita equivalence between
$B$ and $B'$. Suppose that as an $\CO(G'\times G)$-module, $M$ has an
endopermutation source for some (hence any) vertex.  Suppose that $K$
contains a primitive root of order divisible by $|G|$ and $|G'|$.
If $\chi\in \Irr(B)$ and $\chi'\in\Irr(B')$ correspond to each other via the functor
$M\tenB-$, then $c(\chi)_p=c(\chi')_p$.
\end{Corollary}

This will be proved as an application of Theorem \ref{thm1} in 
Section \ref{gendec-Section}. 
We extend this result in two directions (with some overlap): 
to isotypies, and  to stable equivalences  of Morita type between blocks  
given by bimodules with endopermutation source. 

\begin{Theorem}\label{thm15}
Let $B$, $B'$ be blocks of finite group algebras $\OG$, $\OG'$, respectively.
Suppose that $B$ and $B'$ are isotypic.  Let $I : \Irr(B)\cong$ $\Irr(B')$ be
a bijection which is induced by an isotypy between $B$ and $B'$. Then
for every $\chi\in$ $\Irr(B)$ we have  $c(I(\chi))_p=$ $c(\chi)_p$.
\end{Theorem}

This will  be proved in Section \ref{thm1-proof}. By work of Rickard 
\cite[Section 7]{Ricksplendid}, and subsequent generalisations such as  \cite{BoXu}, 
a splendid derived equivalence induces an isotypy,  and hence Theorem
\ref{thm15} applies to character bijections induced by splendid derived equivalences.

\medskip
In order to state the extension to stable equivalenes of Morita type, we need the 
following notation.
 For $B$ a block of $\OG$ we denote by $L^0(B)$ the subgroup of 
$\Z\Irr(B)$ of generalised characters which vanish on $p'$-elements of $G$. It
is well-known that a stable equivalence of Morita type betwee two block algebras 
$B$, $B'$ of finite groups $G$, $G'$ induces an isomorphism $L^0(B)\cong$ $L^0(B')$; 
see e.g.  \cite[Proposition 3.1]{KL10}.

\begin{Theorem} \label{thm2}
Let $G$, $G'$ be finite groups and let $B$, $B'$ be blocks of $\OG$, $\CO G'$,
respectively.  Assume that $K$ contains a primitive root of unity of order divisible
by $|G|$ and $|G'|$.
Let $M$ be an indecomposable  $B'$-$B$-bimodule inducing a stable 
equivalence of Morita type having an endopermutation module $V$ as a source when
regarded as an $\CO(G'\times G)$-module.  Assume that the character of $V$ has
values in $\Z$. Denote by $\Phi_M : \Z\Irr(B)\to$
$\Z\Irr(B')$ the group homomorphism induced by the functor $M\ten_{B}-$.
Then, for any $\psi\in$ $\Z\Irr(B)$ we have 
$$c(\Phi_M(\psi))_p = c(\psi)_p.$$
Suppose moreover that there is an isometry $\Phi : \Z\Irr(B)\cong$ $\Z\Irr(B')$
which extends the isometry $L^0(B)\cong$ $L^0(B')$ induced by $M\ten_B-$.
 Then for any $\chi\in$  $\Irr(B)$ we have
$$c(\Phi(\chi))_p=c(\chi)_p.$$
\end{Theorem}

This will be proved in Section \ref{thm2-proof}.   
By \cite[Proposition 3.3]{KL10},
the isometry $\Phi$, if it exists, is automatically perfect in the sense of 
Brou\'e \cite{Broue88}.   If $V$ has an endosplit $p$-permutation resolution, then
$\Phi$ is induced by a $p$-permutation equivalence, hence by an isotypy, so the
last statement follows in that case also from Theorem \ref{thm15}.
The hypothesis on
 $V$ having character values in $\Z$ is necessary: given an $\OP$-module
 $V $ of $\CO$-rank $1$, tensoring over $\OP$ by the $\OP$-$\OP$-bimodule 
 $\Ind^{P\times P}_{\Delta P}(V)$
 induces a Morita equivalence on the module category of $\OP$. As an 
 $\CO(P\times P)$-module, $\Ind^{P\times P}_{\Delta P}(V)$ has 
 endopermutation source $V$, and the induced Morita equivalence  sends the trivial 
 $\OP$-module $\CO$ to $V$. 
 The equality $1=c(\CO)_p=c(V)_p$ holds precisely when the character values of $V$
 belong to $\Z$. The isometry $\Phi$ exists if the blocks $B$, $B'$ have a cyclic defect 
 group (this goes back to Dade \cite{Dadecyclic}; for the fact that this isometry is
 perfect see  \cite{Licycliccentre} or \cite[Theorem 11.10.2]{LiBookII}).

\begin{Corollary} \label{Cor2}
Let $B$ be a block of $\OG$ with a cyclic defect group $P$.
 Denote by $E\leq \Aut(P)$ an interial quotient of $B$.
Then there is an isotypy between $B$ and $\CO(P\rtimes E)$, and for any such
isotypy, the induced character bijection $\gamma : \Irr(B)\cong$ $\Irr(\CO(P\rtimes E))$
satisfies $c(\chi)_p=$ $c(\gamma(\chi))_p$ for all $\chi\in$ $\Irr(B)$.
\end{Corollary}

Corollary \ref{Cor2} is a special case of a more general scenario
of blocks with an abelian defect group and an inertial quotient which acts freely
on the nontrivial elements of a defect group. For the sake of completeness, 
we state this.

\begin{Corollary}\label{Cor3}
Let $B$ be a block of a finite group algebra $\OG$ with an abelian defect group $P$ 
and inertial quotient $E$ acting freely on $P\setminus \{1\}$. Let $C$ be the
block of $\CO N_G(P)$ which is the Brauer correspondent of $B$.
Assume that $K$ contains a primitive $|G|$-th root of unity.
Suppose that $|\Irr(B)|=|\Irr(C)|$. Then there is a perfect isometry between
$B$ and $C$ such that the induced bijection $\gamma : \Irr(B)\cong$ $\Irr(C)$
satisfies $c(\chi)_p=c(\gamma(\chi))_p$ for all $\chi\in$ $\Irr(B)$.
\end{Corollary}

While both of these two corollaries are special cases of Theorem \ref{thm2},
it should be noted that  for these special cases  it has long
been known that   perfect isometries 
preserve generalised decomposition numbers up to signs, so preserve $p$-parts of
conductors by Theorem \ref{thm1}. We will describe this briefly, including
some pointers to the literature,
at the end of Section \ref{thm2-proof}. 

\begin{Remark}
The Bonnaf\'e-Dat-Rouquier Morita equivalences for $\ell$-blocks 
from  \cite[Theorem 7.5]{BDR}  lift to splendid derived equivalences 
thanks to \cite[Theorem 7.6]{BDR}), and hence the character bijections 
induced by these Morita equivalences  preserve $\ell$-parts of conductors 
by Theorem \ref{thm15}, making use of  the fact that a splendid derived 
equivalence induces an isotypy.   It is not known whether in general these
Morita equivalences are induced by bimodules with endopermutation sources,
so we do not know whether Corollary \ref{Cor1} applies in general to these
Morita equivalences. 
\end{Remark}

\begin{Remark}
Let $G$ be a finite group and $B$ a block of $\OG$ with defect group $P$.
Let $\chi\in$ $\Z\Irr(B)$. We trivially have $c(\chi)_p \geq c(\Res^G_{N_G(P)}(\chi))_p$,
but the above results do not address directly the question when this is an equality. 
There are two issues: first of all, 
$\Res^G_{N_G(P)}(\chi)$  may have components in blocks other than the Brauer 
correspondent $C$, and second, even if one were to truncate the restriction by 
the block idempotent of $C$, this will generally not yield a perfect isometry. 
We mention in Section \ref{Res-Section}  some cases in which the above 
results do yield statements on $p$-parts of conductors under restriction to certain 
subgroups and truncated restriction to some blocks of subgroups.
\end{Remark}

\section{Proof of Theorem \ref{thm1} and Theorem \ref{thm15}}  
\label{thm1-proof}

We use without further comment the well-known fact  that for any two positive 
integers $m$, $n$
we have $\Q(\zeta_m)\cap\Q(\zeta_n)=$ $\Q(\zeta_{\gcd(m,n)})$. Thus if a
set $S$ of cyclotomic integers is contained in $\Q(\zeta_n)$ for some positive
integer $n$, then $c(S)$ divides $n$.

\begin{proof}[Proof of Theorem \ref{thm1}]
Denote by $a$ the non-negative integer satisfying $p^a=c(\chi)_p$.  
The right side in the statement of Theorem \ref{thm1}  is  a power of $p$ 
since all generalised  decomposition numbers are contained in a cyclotomic 
field  of a $p$-power  order root of unity. Denote by $d$ the non-negative integer
such that $p^d= c\left(\{d^u_{\chi,\varphi}\ |\ u\in G_p,\ \varphi\in\IBr(C_G(u))\}\right)$.
We need to show that $a=d$. 

In  the defining equation above of the 
generalised decomposition numbers $d^u_{\chi, \varphi}$, the numbers
$\varphi(s)$ are contained in $\Q(\zeta')$, where $\zeta'$ is a root of unity
of order the $p'$-part of $|G|$. Thus $\chi(us)$ is contained in
$\Q(\zeta_{p^d}, \zeta')$, and hence all  vlaues of $\chi$ are
contained in $\Q(\zeta_{p^dn'})$, where $n'=|G|_{p'}$. 
Since the  values of $\chi$ are also all contained in $\Q(\zeta_{p^am})$
and since $p^am$ is minimal  with this property, it follows that $a\leq d$.

For the other inequality we use Brauer's reciprocity, as descibed in 
\cite[5.14.7, 5.14.8]{LiBookI}. 
Let $u\in G_p$, and denote by $\langle -, -\rangle'$ the
bilinear map on $K$-valued class functions on $p'$-elements of $C_G(u)$ given by
$\langle \alpha, \beta\rangle'=$ $\frac{1}{|C_G(u)|} \sum_{s\in C_G(u)_p'} \ 
\alpha(s)\beta(s^{-1})$, where $\alpha$, $\beta$ are $K$-valued class funtions
defined on $G_{p'}$. Let $\psi\in$ $\IBr(C_G(u)$, and let $\Psi$ be the character
of a projective cover of a simple $kC_G(u)$-module with Brauer character $\psi$.
Brauer's reciprocity, applied to $C_G(u)$, states that for any $\varphi\in$ $\IBr(C_G(u))$ 
the number $\langle \varphi, \Psi\rangle'$ is equal to $1$ if $\psi=\varphi$ and equal to
zero otherwise. Denoting by $d^u_G(\chi)$ the class function on $C_G(u)_{p'}$
defined by $d^u_G(\chi)(s)=\chi(us)$ for all $s\in C_G(u)_{p'}$, it follows from
Brauer's reciprocity that
$$\frac{1}{|C_G(u)|} \sum_{s\in C_G(u)_{p'}} \chi(us) \Psi(s^{-1})\ = \
\langle d^u_G(\chi), \Psi\rangle'\ =\ 
\sum_{\varphi\in\IBr(C_G(u)}\ d^u_{\chi,\varphi} \langle \varphi, \Psi\rangle'\ = \
d^u_{\chi,\psi}.$$
Now the values $\Psi(s^{-1})$ are all contained in $\Q(\zeta')$, so the expression
on the left side shows that the numbers
$d^u_{\chi,\psi}$ are contained in $\Q(\zeta_{p^a}\zeta')$, for
all $u\in G_p$ and all $\psi\in$ $\IBr(C_G(u))$. Since these numbers are all
in $\Q(\zeta_{p^d})$ with $d$ minimal with this property, it follows that
$d\leq a$. 
\end{proof}

As noted earlier, Corollary \ref{Cor05} is an immediate consequence of
Theorem \ref{thm1}. We will give in the next Section an
independent proof of this Corollary, making use of a description of
generalised decomposition numbers due to Puig (cf. \cite[Theorem 5.15.3]{LiBookI}),
which also feeds into the proof of Corollary \ref{Cor1} in the next Section.

\begin{Remark}
With the notation of Theorem \ref{thm1}, suppose that $\chi$ belongs to 
$\Z\Irr(B)$  for some block $B$ of $\OG$.
By Brauer's Second Main Theorem, the numbers
$d^u_{\chi,\varphi}$ are zero unless $u$ belongs to a defect group of $B$ and
$\varphi$ belongs to a block $f$ of $kC_G(u)$ such that $(\langle u\rangle, f)$ is
a $B$-Brauer pair.  The numbers $d^1_{\chi\, \varphi}$ are rational integers. 
\end{Remark}

There are several versions of isotypies in the literature. Brou\'e's original version in
\cite{Broue90} requires the blocks to have isomorphic fusion systems and asks
for   the compatibility of generalised decomposition maps with respect to centralisers 
of cyclic subgroup (as opposed to centralisers of all $p$-subgroups in later versions 
in the literature, such as \cite{BoXu}, which in addition no longer requires the equality of fusion 
systems). Sambale's version \cite[Section 7]{Sa20}  replaces the equality of fusion systems
by a slightly weaker condition, requiring a suitable bijection between conjugacy classes of
Brauer elements, and this is sufficient for the present paper.

\begin{proof}[Proof of Theorem \ref{thm15}]
By definition of an isotypy, there is a bijection $(u,e)\mapsto (u',e')$ between
$G$-conjugacy class representatives of $B$-Brauer elements and
$G'$-conjugacy classes of $B'$-Brauer elements under which $(1_B,B)$ is mapped
to $(1_{B'}, B')$. Let $\chi$ in $\Irr(B)$, and $(u,e)$ a $B$-Brauer element.
By \cite[Proposition 7.2]{Sa20}, the generalised decomposition numbers
$d^{u'}_{I(\chi), \varphi'}$, with $\varphi'\in$ $\IBr(C_{G'}(u'), e')$ are 
$\Z$-linear combinations of the generalised decomposition numbers
$d^u_{\chi, \varphi}$, with $\varphi\in$ $\IBr(C_G(u),e)$. It follows from
Theorem \ref{thm1} that $c(I(\chi))_p\leq c(\chi)_p$. Exchanging the roles of $B$
and $B'$   (or using the fact that the above mentioned coefficient in $\Z$ are part 
of an invertible matrix) shows the equality.
\end{proof}

\section{On  generalised decomposition numbers} \label{gendec-Section}   

We refer to \cite[Section 5.4]{LiBookI} for the Brauer construction, and for
basic terminology regarding Puig's notion of pointed groups we refer
to the exposition in \cite[Section 5.5]{LiBookI}.
Let $G$ be a finite group, $B$ a block of $\OG$, and suppose that $K$ contains
a primitive $|G|$-th root of unity.
 We denote by $\Pr(B)$ the subgroup of $\Z\Irr(B)$ generated
by the characters of projective indecomposable $B$-modules. The subgroup
$L^0(B)$ of $ \Z\Irr(B)$ consists of all generalised characters which are perpendicular
to $\Pr(B)$ with respect to the standard scalar product on $\Z\Irr(B)$. 
We will use   Puig's description 
of generalised decomposition numbers (cf. \cite[Theorem 5.15.3]{LiBookI}):
for $\chi\in$ $\Z\Irr(B)$, $u$ a 
$p$-element in $G$ and $\varphi\in$ $\IBr(\CO C_G(u))$ we have
$$d^u_{\chi,\varphi} = \chi(uj),$$
where $j\in$ $(\OG)^{\langle u\rangle}$ belongs to the local point of $\langle u\rangle$
on $\OG$ with the property that $\Br_{\langle u\rangle}(j)$ is a primitive 
idempotent in $kC_G(\langle u\rangle)$ which does not annihilate the simple 
$kC_G(\langle u\rangle)$-modules with Brauer character $\varphi$. 

\medskip
Denote by $[\OG,\OG]$ the $\CO$-subspace of $\OG$ spanned by all additive 
commutators $xy-yx$,  with $x$, $y\in G$. The space $\OG/[\OG,\OG]$ is
$\CO$-free, dual to $Z(\OG)$, having as a basis the image of a set of representatives 
of the conjugacy classes in $G$ (see e.g.  \cite[Propositions 1.5.4, 1.5.5]{LiBookI}).
The $\CO$-space $\CO G_{p'}$ spanned by $p'$-elements is contained in
the space $X$ spanned by all idempotents in $\OG$, and  these two spaces
have the same image in the space $\OG/[\OG,\OG]$ (cf.  
\cite[Proposition 5.15.16]{LiBookI}).
This image is $\CO$-free and has rank equal to the number of conjugacy classes 
of $p'$-elements, which is in turn equal to the number of conjugacy classes of
 primitive idempotents, hence of irreducible Brauer characters. More precisely, 
conjugate elements in $G$ have the same image in $\OG/[\OG,\OG]$, and
conjugate idempotents in $\OG$ have the same image in $\OG/[\OG,\OG]$, and hence
the image of the space $X$ in $\OG/[\OG,\OG]$ has as $\CO$-bases 
the image of a set of
representatives of $p'$-conjugacy classes in $G$ as well as the image  
of a set of representatives of conjugacy classes of primitive idempotents. 

\medskip
The transition matrix between these two bases involves only $p'$-roots of
unity: given a $p'$-element $s\in G$, the group algebra $\CO\langle s\rangle$ is
commutative, equal to a direct product of copies of $\CO$, and hence
$s$ is an $\CO$-linear combination of idempotents in $\OG$ with
coefficients in $\Q[\sigma]$ for a $p'$-order roots of unity $\sigma$. 
Up to conjugation,
every primitive idempotent in $\OG$ arises with a nonzero coefficient
in some $p'$-element of $G$ in this way.

\medskip
One immediate consequence of Puig's description of generalised decomposition 
numbers   is the following fact.

\begin{Lemma} \label{lem05}
Let $\chi\in$ $\Z\Irr(B)$, $u\in G_p$, and $\varphi\in$  $\IBr(C_G(u))$.
Let $f$ be a primitive idempotent in $\CO C_G(u)$ such that $\CO C_G(u)f$
is a projective cover of a simple $kC_G(u)$-module with Brauer character $\varphi$.
Then $d^u_{\chi, \varphi} = \chi(uf)$.
\end{Lemma}

\begin{proof}
Note that $\Br_{\langle u\rangle}(f)$ is the image of $f$ in $kC_G(u)$, hence
a primitive idempotent. By standard lifting theorems of idempotents, there is
a primitive idempotent $j$ in $\OG^{\langle u\rangle}$ such that
$\Br_{\langle u\rangle}(j)=$ $\Br_{\langle u\rangle}(f)$. 
It follows from  \cite[Theorem 5.12.16]{LiBookI} that we have $\chi(uf)=$
$\chi(uj)=$ $d^u_{\chi, \varphi}$.
\end{proof}

 Lemma \ref{lem05} and the observations above 
 can be used  to give an alternative proof of Corollary \ref{Cor05}.

\begin{proof}[Proof of Corollary \ref{Cor05}]
Let $g\in G$. 
Write $g=us=su$ with $u\in G_p$ and $s\in C_G(u)_{p'}$. 
Since $s$ has order
prime to $p$, it generates, in $\CO C_G(u)$, 
a commutative $\CO$-algebra isomorphic to a direct product of copies of $\CO$.
As mentioned above, $s$ can be written as a linear combination of primitive idempotents
in $\CO C_G(u)$ with cyclotomic coefficients involving only $p'$-order roots of
unity, and moreover, some conjugate of any primitive idempotent in $\CO C_G(u)$ 
arises with a nonzero coefficient for some $p'$element in $C_G(u)$.
 It follows that there is a primitive idempotent $f\in \CO C_G(u)$ such that
$c(\chi(g))_p=$ $c(\chi(uf))$, hence $c(\chi(g))_p\geq$ $c(\chi(uf))$, and this 
holds for all $u$ and all $f$.
There is clearly an element $g\in G$ such that $c(\chi)_p=$ $c(\chi(g))_p$.
Corollary \ref{Cor05} follows now from Lemma \ref{lem05}.
\end{proof}

Similar arguments as in the proof of Lemma \ref{lem05} together  
with Corollary \ref{Cor05} yield the following
observation, stated for future reference.

\begin{Lemma} \label{lem01}
Let $G$ be a finite group, $\chi\in$ $\Z\Irr(G)$, and let  $u\in G_p$.
and any idempotent $j$ in $(\OG)^{\langle u\rangle}$ we have
$c(\chi)_p\geq$ $c(\chi(uj)$. In particular, for any idempotent
$f$ in $\CO C_G(u)$ we have $c(\chi)_p\geq$ $c(\chi(uf))$.
\end{Lemma}

\begin{proof}
For the first statement
we may assume that $j$ is primitive in $(\OG)^{\langle u\rangle}$.
If $\Br_{\langle u\rangle}(j)=0$ then $\chi(uj)=0$ by \cite[Theorem 5.16.12]{LiBookI},
so there is nothing to prove. If $\Br_{\langle u\rangle}(j)\neq 0$ then, again by
Puig's description of generalised decomposition numbers, we have  
$\chi(uj)=$ $d^u_{\chi,\varphi}$, where $\varphi\in$ $\IBr(C_G(u)$ is the Brauer
character of a simple $kC_G(u)$-module which is not annihilated by the
primitive idempotent $\Br_{\langle u\rangle}(j)$ in $kC_G(u)$. Thus the
first statement follows from Corollary \ref{Cor05}. The second statement follows
from the first, since we have $\CO C_G(u)\subseteq$ $(\OG)^{\langle u\rangle}$.
\end{proof}

Aside from Thereom \ref{thm1}, Puig's description of generalised decomposition 
numbers is also the key ingredient for the proof of Corollary \ref{Cor1}. 

\begin{proof}[{Proof of Corollary \ref{Cor1}}]
The generalised decomposition numbers of $\chi$ and $\chi'$ differ at most by
signs; indeed, this follows from combining Puig's description of generalised
decomposition (reviewed in \cite[Theorem 5.15.3]{LiBookI} and above), combined with
\cite[Theorem 7.4.3]{LiBookII}, and \cite[Theorem 9.11.9]{LiBookII}.
Thus Corollary \ref{Cor1} follows from Theorem \ref{thm1}.
\end{proof}

Note that the map sending $v\in$ $\langle u\rangle$ to $\chi(vj)$ is a generalised
character of $\langle u\rangle$, and that if $\chi\in$ $\Pr(B)$,  then this
generalised character belongs to $\Pr(\CO\langle u\rangle)$. These vanish
at all nontrivial elements of $\langle u\rangle$, and thus we have the following
 well-known  consequence that we state for future reference.

\begin{Lemma} \label{lem1}
With the notation above, let $\chi$, $\chi'\in$ $\Z\Irr(B)$ such that $\chi-\chi'\in$
$\Pr(B)$. Then, for any nontrivial $p$-element $u\in G$ and any $\varphi\in$
$\IBr(kC_G(\langle u\rangle)$ we have 
$d^u_{\chi, \varphi} = d^u_{\chi',\varphi}.$
\end{Lemma}

\begin{proof}
A generalised projective character of $\CO\langle u\rangle$ vanishes at all
nontrivial elements in $\langle u\rangle$. Thus, by  the 
preceding observations, if $u$ is a nontrivial $p$-element in $G$,
then $d^u_{\chi-\chi',\varphi}=0$.
\end{proof}

Again for future reference, we mention the following fact which could also be
proved directly (that is, without using Lemma \ref{lem1} and Theorem \ref{thm1}).

\begin{Lemma} \label{lem2}
With the notation above, let $U$ be a finitely generated $\CO$-free
$B$-module, and let $Y$ be a finitely generated projective $B$-module.
Denote by $\chi$ and $\psi$ the characters of $U$ and of $U\oplus Y$,
respectively. 
Then $c(\chi)_p=$ $c(\psi)_p$.
\end{Lemma}

\begin{proof}
By Lemma \ref{lem1} the characters of $U$ and of $U\oplus Y$ have the same
non-ordinary generalised decomposition numbers. Thus Lemma \ref{lem2}
follows from Theorem \ref{thm1}.
\end{proof}

\section{Quoted results  on  endopermutation modules}

Background material on
$p$-permutation modules can be found for instance in \cite[Section 5.11]{LiBookI}.
A finitely generated $\OG$-module $U$ is called a {\em $p$-permutation module}
if its restriction to some  Sylow $p$-subgroup of $G$ is a permutation
module, or equivalently, if its restriction to any $p$-subgroup is a permutation
module. Any direct summand of a $p$-permutation module is again a 
$p$-permutation module. In particular, a block algebra $B$ of $\OG$ is
a $p$-permutation $\CO(G\times G)$-module. For any $p$-subgroup $Q$ of
$G$ and any idempotent $j\in$ $(\OG)^Q$ the $\CO(Q\times G)$-module
$j\OG$ is a $p$-permutation module. This module is indecomposable if and
only if $j$ is primitive in $(\OG)^Q$ (that is, $j$ belongs to a point of $Q$ on $\OG$).
In that case $\Delta Q$ is a vertex of $i\OG$ if and only if $\Br_Q(j)\neq 0$
(that is, $j$ belongs to a local point of $Q$ on $\OG$). The following 
well-known easy observation will be used in the proof of Lemma \ref{lem75}.

\begin{Lemma} \label{lem25}
Let $G$ be a finite group, $Q$ a $p$-subgroup of $G$, and $j$ an idempotent
in $(\OG)^Q$ such that $\Br_Q(j)\neq 0$. Then $j\OG j$ is a permutation
$\CO(Q\times Q)$-module  which is projective as a left and as a right
$\OQ$-module, and the following are equivalent.
\begin{itemize}
\item[{\rm (i)}]  $\Br_Q(j) \neq 0$.
\item[{\rm (ii)}] As an $\CO(Q\times Q)$-module, $j\OG j$ has an indecomposable
direct summand with vertex $\Delta Q$.
\item[{\rm (iii)}] As an $\CO(Q\times Q)$-module, $j\OG j$ has an indecomposable
direct summand isomorphic to $\OQ$.
\end{itemize}
\end{Lemma}

\begin{proof}
This  is, for instance, an immediate consequence of \cite[Lemma 5.8.8]{LiBookI}.
\end{proof}

\medskip
For basics on endopermutation modules, a concepet introduced by Dade
 \cite{Dadeendo}, we refer to the exposition in 
\cite[Section 7.3]{LiBookII} as well as to the broader introduction to this
topic, with further references,  in \cite{Thevtour}.
Given a finite $p$-group $P$ and an 
indecomposable endopermutation $\OP$-module $V$ with vertex $P$,
for any subgroup $Q$ of $P$ the $\OQ$-module $\Res^P_Q(V)$ is an
endopermutation $\OQ$-module with a direct summand $V_Q$ having
$Q$ as a vertex. Up to isomorphism, $V_Q$ arises with a multiplicity 
congruent to either $1$ or $-1$ modulo $p$ in any decomposition
of $\Res^P_Q(V)$ as a direct sum of indecomposable $\OQ$-modules.
All indecomposable direct summands of $\Res^P_Q(V)$
not isomorphic to $V_Q$ have vertices strictly smaller than $Q$.
For $\CF$ a fusion system on $P$, we say that $V$ is $\CF$-{\em stable}
if for any isomorphism $\varphi : Q\to R$ in $\CF$ we have an
isomorphism of $\OQ$-modules $V_Q\cong$ $\Res_\varphi(V_R)$
(see \cite[Section 9.9]{LiBookII} for more details).  
We will say that $V$ is defined over the unramified subring  of $\CO$ 
 if $V\cong$ $\CO\ten_{\CO'} V'$ for some
$\CO'$-module $V'$, where $\CO'$ is the unramified complete discrete valuation
subring with residue field $k$,  and then $V'$ is again an indecomposable endopermutation 
$\CO' P$-module with vertex $P$. For any $V$ as before, there is an endopermutation
$\OP$-module $W$ such that $k\tenO W\cong$ $k\tenO V$ and such that $W$ is
defined over $\CO'$. That is, up to modifying $V$ by a rank one $\OP$-module we may
assume that $V$ is defined over $\CO'$. If $V$ is defined over $\CO'$, then the
character $\chi_V$ of $V$ takes values in $\Z$. 
If moreover $P$ is cyclic and $u$ a generator of $P$, then $\chi_V(u)\in$ $\{1,-1\}$.
The following is well-known.

\begin{Lemma} \label{lem3}
Let $P$ be a $p$-subgroup of $G$, and let $V$ be an indecomposable
endopermutation $\OP$-module with vertex $P$ defined over the unramified
subring of $\CO$.
\begin{itemize}
\item[{\rm (i)}] 
Let $U$ be an finitely generated $\CO$-free indecomposable $\OP$-module 
with vertex $P$.
Then $V\tenO U$ has a unique indecomposable direct summand with vertex $P$,
and all other indecomposable summands have vertices strictly smaller than $P$, in
any decomposition of $V\tenO U$ as a direct sum of indecomposable
$\OP$-modules with respect to the diagonal action of $P$.
\item[{\rm (ii)}] 
Let $Y$ be a finitely generated $\CO$-free indecomposable 
$\OP$-$\OG$-module which has vertex $\Delta P$ as an $\CO(P\times G)$-module.
Consider $V\tenO Y$ as an $\CO(P\times G)$-module  with $P$ acting diagonally on 
the left.
Then $V\tenO Y$ has a unique indecomposable direct summand with vertex $\Delta P$,
and all other indecomposable summands have vertices of order strictly smaller 
than $|P|$, in
any decomposition of $V\tenO Y$ as a direct sum of indecomposable
$\CO(P\times G)$-modules, 
\end{itemize}
\end{Lemma}

\begin{proof} 
See e.g. \cite[Proposition 7.3.10, 7.3.18]{LiBookII} for proofs of these two statements.
\end{proof}

We spell out two  observations that may look slightly artificial but 
are taylor-made for the proof of Theorem
\ref{thm2} below.

\begin{Lemma} \label{lem5}
Let $Q$ be a $p$-subgroup of $G$ and $R$ a proper subgroup of $Q$.

\begin{itemize}
\item[{\rm (i)}] Let $N$ be a finitely generated $\OR$-$\OQ$-bimodule.
Then $\Delta Q$ is not contained in any vertex of any indecomposable  direct
summand of $\OQ\tenOR N$ regarded as an $\CO(Q\times Q)$-module.

\item[{\rm (ii)}] Let $X$ be a finitely generated $\OQ$-$\OG$-bimodule, and let
$V$ be a finitely generated $\OQ$-module with vertex $R$. Consider $V\tenO X$
as an $\OQ$-$\OG$-bimodule with $Q$ acting diagonally on the left and $G$
acting on the right via the action induced by that on $X$. Then
$\Delta Q$ is not contained in any vertex of any indecomposable direct
summand of $V\tenO X$ regarded as an $\CO(Q\times G)$-module.

\end{itemize}
\end{Lemma}

\begin{proof}
We have an obvious identification $\OQ\tenOR N=$ $\Ind^{Q\times G}_{R\times G}(N)$.
This implies that no indecomposable direct summand of this module has a vertex containing
$\Delta Q$. Thus (i) follows.   
For (ii), by the assumptions,
$V$ is isomorphic to a direct summand of $\Ind^Q_R(T)$ for some finitely generated
$\OR$-module $T$. Thus $V\tenO X$ is isomorphic to a direct summand of
$\Ind^Q_R(T)\tenO X\cong$ $\Ind^{Q\times G}_{R\times G}(T\tenO X)=$
$\OQ\tenOR (T\tenO X)$. Thus (ii) follows from (i).
\end{proof}

\begin{Remark} 
We have stated Lemma \ref{lem5} in the form needed later, but one easily 
sees that this Lemma holds more generally with $\Delta Q$ replaced
by a twisted diagonal subgroup of the form $\Delta_\varphi(Q)=$
$\{(u, \varphi(u))\ |\ u\in Q\}$, where in (i) $\varphi$ is any automorphism of
$Q$ and in (ii) $\varphi : Q\to G$ is any  injective group homomorphism.
\end{Remark}

\section{Proof of Theorem \ref{thm2}} \label{thm2-proof}

We use the notation and hypotheses from Theorem \ref{thm2}; that is,
$G$, $G'$ are finite groups, $B$, $B'$ are blocks of $\OG$, $\CO G'$,
respectively, and we assume that $K$ contains a primitive root of unity of 
order divisible by $|G|$ and $|G'|$.
Let $M$ be an indecomposable  $B'$-$B$-bimodule inducing a stable 
equivalence of Morita type having an endopermutation module as a source when
regarded as an $\CO(G'\times G)$-module. Denote by $\Phi_M : \Z\Irr(B)\to$
$\Z\Irr(B')$ the group homomorphism induced by the functor $M\ten_{B}-$.
Since $M$ induces a stable equivalence of Morita type, by definition,
$M$ is finitely generated projective as a left $B'$-module and as a right
$B$-module, and we have a $B'$-$B'$-bimodule isomorphism
$$M\tenB M^*\cong B'\oplus X'$$ 
for some projective $B'$-$B'$-bimodule $X'$, and a $B$-$B$-bimodule isomorphism
$$M^*\ten_{B'} M \cong B\oplus X$$
for some projective $B$-$B$-bimodule $X$. We will later make use of the fact that 
$B\oplus X$  is a $p$-permutation $\CO(G\times G)$-module.
Since $M$ has an endopermutation module
as a source, it follows from results of Puig \cite[7.6]{Puigbook} (described in 
\cite[Theorem 9.11.2]{LiBookII}) that we may assume that $B$, $B'$
have a common defect group $P$, and that the diagonal subgroup
$\Delta P=\{(u,u)\ | \ u\in P\}$ of $G'\times G$ is a vertex of $M$.
Moreover, there are source idempotents $i\in B^P$, $i'\in (B')^P$
which determine the same fusion system $\CF$ on $P$, 
such that $M$ is isomorphic to a direct summand of
$$B'i' \tenOP \Ind^{P\times P}_{\Delta P}(V) \tenOP iB$$
and such that $V$, regarded as $\OP$-module via the
canonical isomorphism $P\cong$ $\Delta P$, is  an indecomposable
$\CF$-stable endopermutation module with full vertex $P$.
Set $A=iBi$ and $A'=i'B'i'$; these are source algebras of $B$, $B'$ respectively.
The $A'$-$A$-bimodule  $i'Mi$ is isomorphic to a direct summand of
$$A' \tenOP \Ind^{P\times P}_{\Delta P}(V) \tenOP A$$
and induces a stable equivalence of Morita type between $A$ and $A'$.

\begin{Lemma} \label{lem6}
Let $Q$ be a subgroup of $P$. Let $V_Q$ be an indecomposable direct
summand of $\Res^P_Q(V)$ with vertex $Q$. Let $j$ be a primitive local
idempotent in $A^Q$. The $\OQ$-$B$-bimodule $V\tenO jB$, when
regarded as an $\CO(Q\times G)$-module, has
up to isomorphism a unique direct summand with vertex $\Delta Q$.  
This summand is isomorphic to a direct summand of $V_Q\tenO j\OG$, and
$V_Q$ is a source of this summand.  Any other indecomposable direct 
summand of $V\tenO j\OG$ has vertices  of order strictly smaller that $|Q|$.
\end{Lemma}

\begin{proof}
This follows from  Lemma \ref{lem3}; the statement on $V_Q$ being
a source follows from the fact that $jB$ itself  is a
trivial source $\CO(Q\times G)$-module (with vertex $\Delta Q$).
\end{proof}

\begin{Lemma} \label{lem7}
Let $Q$ be a nontrivial subgroup of $P$. 
Let $j'$ be a primitive local idempotent in $(A')^Q$. 
The $\OQ$-$B$-bimodule $j'M$ has a unique non-projective indecomposable 
summand in any decomposition of $j'M$ as a direct sum
of indecomposable $\OQ$-$B$-bimodules. 
\end{Lemma}

\begin{proof}
By a result of Brou\'e (see e.g.  \cite[Proposition 2.17.11]{LiBookI}),
the functors $-\ten_{B'}M$ and $-\ten_B M^*$ induce an equivalence between
the stable categories of perfect $\OQ$-$B$-modules and perfect $\OQ$-$B'$-modules.
We have $j'M'\tenB M^*\cong$ $j'B'\oplus j'X'$,
and hence $j'B'$ is, up to isomorphism the unique non-projective
$\OQ$-$B'$-bimodule summand on the right side. Thus $j'M'$ itself has exactly one
non-projective $\OQ$-$B'$-bimodule summand in any direct sum decomposition.
The Lemma follows.
\end{proof}

\begin{Lemma} \label{lem75}
Let  $Q$ be a fully $\CF$-centralised nontrivial subgroup of $P$.
For any primitive local idempotent $j'$ in $(A')^Q$ there is a primitive local
idempotent $j$ in $A^Q$ such that  the up to isomorphism 
unique indecomposable direct summand of the $\OQ$-$B$-module 
$V_Q\tenO jB$ with vertex $\Delta Q$ is isomorphic to the non-projective 
indecomposable direct summands of $j'M$.
\end{Lemma}

\begin{proof}
The assumptions on $M$ imply  that $j'M$ is isomorphic to
a direct summand of the $\OQ$-$B$-bimodule
$j'B'i'\tenOP \Ind^{P\times P}_{\Delta P}(V)\tenOP iB$.
Thus an indecomposable non-projective summand $N$ of $j'M$ is
isomorphic to an indecomposable $\OQ$-$B$-bimodule summand of 
$W \tenOP \Ind^{P\times P}_P(V)\tenOP iB$ for some indecomposable
$\OQ$-$\OP$-bimodule summand $W$ of $j'B'i'$. By \cite[Theorem 8.7.1]{LiBookII},
$W\cong$ $\OQ\tenOR ({_\psi{\OP}})$ for some injective group homomorphism
$\psi : \OR\to$ $\OP$ in the fusion system $\CF$. Here, and similarly later, 
${_\psi\OP}$ denotes
the $\OR$-$\OP$-bimodule which is equal to $\OP$ as a right $\OP$-module,
with $v\in R$ acting on the left as multiplication by $\psi(v)$.
Recall that $M\tenB M^*\cong$ $B'\oplus X'$ for some projective
$B'$-$B'$-bimodule, so this is a $p$-permutation $\CO(G\times G)$-module,
and hence $j'M\tenB M^*j'$ is a permutation $\CO(Q\times Q)$-module.
Now $j'M\tenB M^*j'$ has $j'B'j'$ as a direct $\CO(Q\times Q)$-module
summand, hence, by Lemma \ref{lem25},  has a direct $\CO(Q\times Q)$-module 
summand isomorphic  to $\OQ$, and this summand has vertex  $\Delta Q$.
Lemma \ref{lem5} implies that $R=Q$,
and that the vertices of $N$ have order $|Q|$. Thus $W\cong$ ${_\psi{\OP}}$ for
some morphism $\psi : Q\to P$ in $\CF$, and hence
the summand $N$ of $j'M$ is isomorphic to a direct summand of 
${_\varphi{\Ind^{P\times P}_{\Delta P}(V)\tenOP iB}}$. Using the
indecomposability of $N$ we get that $N$ is isomorphic to a direct
summand of ${_\varphi{(V_{\varphi(Q)})\tenO nB}}$ for some primitive
 idempotent $n$ in $A^{\varphi(Q)}$. Since $N$ has vertices of order $|Q|$, the
 idempotent $n$ belongs to a local point of $Q$ on $A$.
The fusion stability of $V$ and \cite[Theorem 8.7.4]{LiBookII}
 imply together  that $N$ is
isomorphic to a direct summand of $V_Q\tenO jB$ for some
primitive local idempotent $j$ in $A^Q$. The result follows.
\end{proof}

\begin{Remark}
One can show that the correspondence $j'\mapsto j$ induces a bijection
between the sets of local points of $Q$ on $A'$ and on $A$ which is
compatible with fusion in $A$ and $A'$. This
can be used to give an alternative proof of a result of Puig
stating that $A$ and $A'$ have the same categories of nontrivial local pointed
groups (which is slightly stronger than merely stating that the
fusion systems are preserved under stable equivalences of Morita
type with endopermutation source). 
\end{Remark}

\begin{Lemma} \label{lem8}
With the notation above, assume that the character $\chi_V$ of $V$ as values
in $\Z$. Let $U$ be a finitely generated $\CO$-free
$B$-module.  Denote by $\chi$ the character of $U$ and by $\psi$ the
character of $M\tenB U$. Then, up to signs and up a suitable order,
the non-ordinary generalised deomposition numbers of $\chi$ and
$\psi$ coincide.
\end{Lemma}

\begin{proof}
We note first that
it suffices to show that every non-ordinary generalised decomposition
number of $\psi$ is up to sign  a non-ordinary generalised  
decomposition
number of $\chi$. Indeed, if this statement holds in general, then 
that same statement applied to the $B'$-module $M\tenB U$ and 
the $B$-$B'$-bimodule $M^*$ shows 
that every non-ordinary generalised decomposition number of $M^*\ten_{B'} M\tenB U$ is
also, up to a sign, a non-ordinary generalised decmposition number of
$M\tenB U$. Since $M^*\ten_{B'} M\tenB U\cong$ $U\oplus T$ for some projective
$B$-module $T$ it follows from Lemma \ref{lem1} that  the non-ordinary
generalised decomposition numbers of $U$ and of $M^*\ten_{B'} M\tenB U$
coincide, and hence up to indexing, coincide with those of $M\tenB U$ as well.

Let $u$ be a nontrivial $p$-element in $P$. Set $Q=\langle u\rangle$.
Since $\chi$ and $\psi$ are class functions, we may assume that $Q$ is
fully $\CF$-centralised.

Note that the character of any finitely generated $\CO$-free indecomposable 
$\OQ$-module with vertex a proper subgroup of $Q$ vanishes at $u$.
if $j$ be a primitive local idempotent in $A^Q$, then the generalised 
decomposition number of $\chi$ at $u$ and the irreducible Brauer character of $kC_G(u)$ 
determined by $j$ is equal to $\chi(uj)$. Similary, if $j'$ is a primitive local
idempotent in $(A')^Q$ then the generalised decomposition number  of
$\psi$ at $u$ and the irreducible  Brauer character  of $kC_{G'}(u)$ determined
by $j'$ is equal to $\psi(uj')$. 

Fix a primitive local idempotent $j'$ in $(A')^Q$. By the above, we need to 
show that there is a primitive local idempotent $j$ in $A^Q$ such that
$\chi(uj)=$ $\pm\psi(uj')$.

By Lemma \ref{lem75}  there is a primitive local
idempotent $j$ in $A^Q$ such that  the up to isomorphism 
unique indecomposable direct summand of the $\OQ$-$B$-module 
$V_Q\tenO jB$ with vertex $\Delta Q$ is isomorphic to the non-projective 
indecomposable direct summands of $j'M$.

Since the character of an $\OQ$-module vanishes at $u$ unless possibly
if $u$ belongs to a vertex of a summand of this module, it follows that
the characters of the left $\OQ$-modules $j'M\tenB U$ and
$V_Q\tenO jU$ coincide when evaluated at $u$. Moreover, the character of $V_Q$
evaluated at $u$ is in $\{1,-1\}$. (We use here the hypothesis that $V$  has a 
$\Z$-valued character.) Thus the evaluations at $u$ of the characters
of $V_Q\ten jU$ and $jU$ differ at most by a sign.
Together it follows that the generalised decomposition numbers
$\chi(uj)$ and $\psi(uj')$   differ at most by a sign. The result follows.
\end{proof}

\begin{proof}[Proof of Theorem \ref{thm2}]
By Lemma \ref{lem8},  given an $\CO$-free $B$-module of finite
$\CO$-rank, the characters of $U$ and of $M\tenB U$ have the same
generalised non-ordinary decomposition numbers. Thus the first
statement follows from Theorem \ref{thm1}.
Let $\psi\in$ $\Z\Irr(B)$.  Denote by $-\cdot_B-$ the maps on Grothendieck
groups induced by the functor $-\tenB-$. 
By \cite[Proposition 3.3]{KL10} the generalised character
 $\sum_{\chi\in\Irr(B)}\ \Phi(\chi) \chi^*  - \chi_M$ is in fact a generalised
projective character of $G'\times G$. Here $\chi^*$ is defined by
$\chi^*(g)=\chi(g^{-1})$ for all $g\in G$. 
Thus  
$$\Phi(\psi)-\Phi_M(\psi) = (\sum_{\chi\in\Irr(B)}\ \Phi(\chi) \chi^*  - \chi_M)\cdot_B \psi$$
is a generalised projective character in $\Z\Irr(B')$. 
Therefore, if $u$ is a nontrivial $p$-element in $G'$ and $j$ an idempotent belonging to
a local point of $\langle u\rangle$ on $B'$, then this character evaluated at elements
of the form $vj$, with $v\in \langle u\rangle$ 
yields a generalised projective character of $\CO\langle u\rangle$, hence vanishes
for $u$ a nontrivial $p$-element in $G'$. Thus the generalised characters
$\Phi(\psi)=(\sum_{\chi\in\Irr(B)}\ \Phi(\chi) \chi^*)\cdot_B \psi$ and
$\Phi_M(\psi) = \chi_M\cdot_B \psi$ have the same generalised non-ordinary
decomposition numbers. As before, Theorem \ref{thm1} implies that they have 
conductors whose $p$-parts are equal.
Combined with the first statement and 
applied with $\chi\in\Irr(B)$ instead of $\psi$  implies 
the last statement in Theorem \ref{thm2}.
\end{proof}

As mentioned in the Introduction, the Corollaries \ref{Cor2} and \ref{Cor3}
can  be played back more directly to Theorem \ref{thm1} by making use of
long standing results on generalised decomposition numbers in these cases;
we give some references in the proofs below, which circumvent the need to use 
Theorem \ref{thm2}. In the case of blocks with a cyclic defect group, the 
generalised decomposition numbers are due to Dade \cite{Dadecyclic}, 
except for the explicit connection between the signs that appear there on the one hand, 
and the endopermutation module $V$ and perfect isometry on the other hand.

\begin{proof}[Proof of Corollary \ref{Cor2}]
Assume that $P$ is cyclic. By \cite[Th\'eor\`eme (13.2), Th\'eor\`eme (13.6)]{Lithese}
there is an isometry such that the non-ordinary generalised decomposition numbers
of corresponding irreducible characters are equal up to signs. This isometry is
automatically perfect by \cite[Proposition 3.3]{KL10}. The fact that this isometry
is perfect  is also shown  in \cite{Licycliccentre};  see 
\cite[Theorem 11.1.14]{LiBookII} for a broader expository account.
The result follows from Theorem \ref{thm1}.
\end{proof}

 \begin{proof}[Proof of Corollary \ref{Cor3}]
Note that $C$ has $\CO(P\rtimes E)$ as source algebra. Thus the hypotheses
and \cite[Theorem 5.10.5]{LiBookII} imply the existence of a perfect isometry
which preserves generalised decomposition numbers up to signs.
Thus the result follows again from Theorem \ref{thm1}.
\end{proof}

\section{Some remarks on conductors and restriction to blocks of subgroups}
\label{Res-Section}

Restriction and induction functors between blocks of finite groups
are given by bimodules with trivial source, so they differ in general from the
stable equivalences of Morita type with endopermutation source considered
in the main results of this paper. We mention some cases where we do get
results on restrictions and truncated restrictions.

\begin{Proposition} \label{PropIntro1}
Let $G$ be a finite group. Assume that $K$ contains a primitive $|G|$-th root 
of unity.  Let $\chi\in$ $\Z\Irr(G)$, and let  $u\in G_p$ and $\varphi\in$
$\IBr(C_G(u))$ such that $c(\chi)_p=c(d^u_{\chi,\varphi})$.
Let  $H$ be a subgroup of $G$ containing $C_G(u)$. We have
$c(\chi)_p =$ $c(\Res^G_H(\chi))_p$.
\end{Proposition}

\begin{proof}
 Clearly $c(\Res^G_H(\chi))_p\leq$
$c(\chi)_p$. For the other inequality, note that Lemma \ref{lem05} and 
Theorem \ref{thm1} imply that $c(\chi)_p$ is the maximum of the numbers
$\chi(vf)$, with $v\in G_p$ and $f$ running over the primitive idempotents
in $\CO C_G(v)$. Applied to $v=u$ and using the hypotheses, there is
a primitive idempotent $f\in$ $\CO C_G(u)$ such that
$c(\chi)_p=$ $c(d^u_{\chi,\varphi}) =$ $c(\chi(uf)) \leq$ $c(\Res^G_H(\chi))_p$.
\end{proof}

Given a generalised character $\chi$  of a finite group algebra $\OG$
and a central idempotent $b$ in $Z(\OG)$ we denote by $b\cdot \chi$ the
component of $\chi$ belonging to $\OGb$; that is, $(b\cdot \chi)(g)=$ 
$\chi(bg)$  for all $g\in G$. 

\begin{Lemma} \label{lem10}
Let $G$ be a finite group and let $b$, $b'$ be orthogonal idempotents
in $Z(\OG)$. Let $\chi\in$ $\Z\Irr(G)$ such that $\chi=$ $(b+b')\cdot \chi$.
We have  $c(\chi)_p =$ $\max\{ c(b\cdot \chi)_p, c(b'\cdot\chi)_p\}$. 
\end{Lemma} 

\begin{proof}
The generalised decomposition numbers are partitioned according to
blocks (by Brauer's Second Main Theorem, which is implied directly by Puig's
description of these numbers). Thus Theorem \ref{thm1} implies the Lemma.
\end{proof}

\begin{Lemma} \label{lem11}
Let $G$ be a finite group, $B$ a block of $\OG$, $P$ a defect group of
$B$, and $H$ a subgroup of $G$ containing $N_G(P)$. Let $C$ be the block
of $\OH$ with $P$ as defect group which is the Brauer correspondent of $B$.
Let $\chi\in$ $\Z\Irr(B)$. If $c(\chi)_p=$ $c(1_C\cdot \Res^G_H(\chi))_p$,
then $c(\chi)_p=$ $c(\Res^G_H(\chi))_p$.
\end{Lemma}

\begin{proof}
We clearly have $c(\chi)_p \geq$ $ c(\Res^G_H(\chi))_p \geq$
$  c(1_C\cdot \Res^G_H(\chi))_p$, where the last inequality 
follows from Lemma \ref{lem10}. The result follows from the 
hypotheses.
\end{proof}

\begin{Proposition}\label{TI-Prop}
Let $G$ be a finite group, $B$ a block of $\OG$, and $P$ a defect group 
of $B$. Suppose that $H$ is a subgroup of $G$ containing $P$ such that
for all $g\in $ $G\setminus H$ we have $P\cap {^gP} = \{1\}$.
Then for every $\chi\in$ $\Z\Irr(B)$ we have 
$$c(\chi)_p = c(\Res^G_H(\chi))_p=  c(1_C\cdot \Res^G_H(\chi))_p.$$
\end{Proposition}

\begin{proof}
It is well-known (see e.g. \cite[Theorem 9.8.6]{LiBookII}) that the hypotheses 
imply that the truncated induction functor $1_B\cdot\Ind^G_H$ and the truncated 
restriction functor $1_C\cdot\Res^G_H$  
given by the $B$-$C$-bimodule $B\cdot 1_C$ and its dual induces a stable
equivalence of Morita type whose indecomposable summands have trivial 
source. Thus Theorem \ref{thm2} implies that 
$c(\chi)_p =$ $ c(1_C\cdot \Res^G_H(\chi))_p$.
We clearly have $c(\chi)_p\geq$ $c(\Res^G_H(\chi))_p$.
Since $H$ contains $C_G(u)$ for any nontrivial element $u$ in $P$, 
we also have $c(\Res^G_H(\chi))_p\geq$
$c(1_C\cdot \Res^G_H(\chi))_p$, by Lemma \ref{lem10}. 
The result follows.
\end{proof}

The following  Corollary is a special case of \cite[Theorem C]{HungFry}
(and is in  one of the steps in the proof of \cite[Theorem 3.4]{HungFry});
we mention this just  to illustrate the use of Theorem \ref{thm2}.

\begin{Corollary} \label{cyclic-Cor}
Let $G$ be a finite group, $B$ a block of $\OG$ with a nontrivial cyclic 
defect group $P$. Denote by $P_1$ the unique subgroup of order $p$ of $P$
and set $H=N_G(P_1)$. Denote by $C$ the block of $\OH$ which is the Brauer
correspondent of $B$. There is a bijection $\gamma : \Irr(B)\to$ $\Irr(C)$ 
and there are signs $\delta_\chi\in$ $\{1,-1\}$ such that 
$\Z\Irr(B)\cong$ $\Z\Irr(C)$ sending $\chi\in$ $\Irr(B)$ to $\delta_\chi\gamma(\chi)$
is a perfect isometry, and such that
$$c(\chi)_p= \Res^G_H(\chi))_p =  c(1_C\cdot \Res^G_H(\chi))_p = c(\gamma(\chi))_p.$$
\end{Corollary}

\begin{proof}
It is well-known and easy to check that for $g\in G\setminus H$ we have
$P\cap {^gP} = \{1\}$. Moreover, by
\cite[Theorem 11.10.2]{LiBookII}, 
there is a bijection $\gamma : \Irr(B)\to$ $\Irr(C)$ such that
$1_C\cdot \Res^G_H(\chi) =$ $\delta_\chi\gamma(\chi) + \psi_\chi$ for some
projective character $\psi$  in $\Pr(C)$ and some $\delta_\chi\in$ $\{1,-1\}$.
Thus the Corollary follows from Proposition \ref{TI-Prop} and Theorem \ref{thm2}.
\end{proof}

\bigskip\noindent{\bf Acknowledgements.} 
The author wishes to thank Hung Ngoc Nguyen and  Radha Kessar
for some very helpful comments and suggestions.
Moreover, the author acknowledges funding from EPSRC grant EP/X035328/1.


\end{document}